\documentclass{amsart}

\usepackage{tikz-cd}
\usepackage[colorlinks=false]{hyperref}
\usepackage{amsmath, amsbsy, amssymb, amsthm, thmtools}
\usepackage[top=2.74cm,bottom=2.74cm,left=2.5cm,right=2.5cm,marginparwidth=2cm, marginparsep=3mm]{geometry}
\usepackage{enumitem}
\usepackage{graphicx}
\usepackage{todonotes}
\usepackage{floatrow}
\usepackage[labelformat=empty]{caption}

\newcommand{\upset}{\ensuremath{\mathord{\uparrow}\mkern1mu}} 
\newcommand{\downset}{\ensuremath{\mathord{\downarrow}\mkern1mu}} 

\author{Marco Abbadini}
\address{Dipartimento di Matematica {\sl Federigo Enriques}, Universit\`a degli Studi di Milano, via Cesare Saldini 50, 20133 Milano, Italy}
\email{marco.abbadini@unimi.it}
\author{Luca Reggio}
\address{Department of Computer Science, University of Oxford, UK}
\email{luca.reggio@cs.ox.ac.uk}

\thanks{To appear in \emph{Applied Categorical Structures}}
\title{\vspace*{-1.1em}On the axiomatisability of the dual of compact ordered spaces}

\theoremstyle{plain}
\newtheorem{theorem}{Theorem}
\newtheorem{proposition}[theorem]{Proposition}
\newtheorem{lemma}[theorem]{Lemma}
\newtheorem{corollary}[theorem]{Corollary}

\newtheorem*{claim*}{Claim}

\theoremstyle{definition}
\newtheorem{definition}[theorem]{Definition}

\newtheorem{remark}[theorem]{Remark}

\newtheorem*{notation}{Notation}

\newcommand{\adj}{\raisebox{.07\baselineskip}{\ensuremath{\scriptscriptstyle \top}}}
\newcommand{\refl}{\rho} 
\newcommand{\op}{\mathrm{op}} 

\renewcommand{\epsilon}{\varepsilon}

\renewcommand{\theta}{\vartheta}
\renewcommand{\phi}{\varphi}

\newcommand{\les}{\preccurlyeq}

\newcommand{\seq}{\subseteq}
\DeclareMathOperator{\image}{Im}
\newcommand{\fix}{\Phi}

\newcommand{\Quot}{\mathbf{Q}} 
\newcommand{\Preo}{\mathbf{P}} 

\DeclareMathOperator{\KH}{\mathsf{KH}} 
\DeclareMathOperator{\PreC}{\mathsf{KH}_{\scriptsize \les}} 
\DeclareMathOperator{\PC}{\mathsf{KH}_{\scriptsize \leq}} 
\DeclareMathOperator{\PCop}{\mathsf{KH}_{\scriptsize \leq}^\op}
\DeclareMathOperator{\Set}{\mathsf{Set}}
\renewcommand{\int}{[0,1]}

\renewcommand{\leq}{\leqslant}
\newcommand{\epi}{\twoheadrightarrow}
\newcommand{\mono}{\rightarrowtail}
\newcommand{\rmono}{\hookrightarrow}

\newcommand{\C}{\mathsf{C}}

\newcommand{\F}{\mathsf{F}}
\renewcommand{\Pr}{\mathsf{PSp}}

\newcommand{\SP}{{\rm SP}}

\begin{document}

\maketitle

\vspace{-0.8cm}
\begin{abstract}
We provide a direct and elementary proof of the fact that the category of Nachbin's compact ordered spaces is dually equivalent to an $\aleph_1$-ary variety of algebras. Further, we show that $\aleph_1$ is a sharp bound: compact ordered spaces are not dually equivalent to any $\SP$-class of finitary algebras. 
\end{abstract}

\vspace{0.5cm}
In 1936, in his landmark paper~\cite{Stone1936}, M.\ H.\ Stone described what is nowadays known as Stone duality for Boolean algebras. In modern terms, it states that the category of Boolean algebras with homomorphisms is dually equivalent to the category of totally disconnected compact Hausdorff spaces with continuous maps. If we drop the assumption of total disconnectedness, we are left with the category $\KH$ of compact Hausdorff spaces and continuous maps. Duskin showed in 1969 that the opposite category $\KH^\op$ --- which, by Gelfand-Naimark duality~\cite{GN1943}, can be identified with the category of commutative unital $\mathrm{C}^*$-algebras --- is monadic over the category of sets and functions \cite[{}5.15.3]{Duskin1969}. In fact, $\KH^{\op}$ is equivalent to a variety of algebras. Although not finitary, this is an $\aleph_1$-ary variety. That is, it can be described by operations of at most countably infinite arity. A generating set of operations was exhibited by Isbell~\cite{Isbell1982}, while a finite axiomatisation of this variety was provided in~\cite{MarraReggio2017}. Therefore, if we allow for infinitary operations, Stone duality for Boolean algebras can be lifted to compact Hausdorff spaces, retaining the algebraic nature.

Shortly after his paper on the duality for Boolean algebras, Stone published a generalisation of this theory to distributive lattices~\cite{Stone1938}. In his formulation, the dual category consists of the nowadays called spectral spaces and perfect maps. While spectral spaces are in general not Hausdorff, H.\ A.\ Priestley showed in 1970 that they can be equivalently described as certain compact Hausdorff spaces equipped with a partial order relation~\cite{Priestley1970}. More precisely, Priestley duality states that the category of (bounded) distributive lattices is dually equivalent to the full subcategory of Nachbin's compact ordered spaces on the totally order-disconnected objects (cf.\ Definitions~\ref{def:compact-pospace} and~\ref{def:Priestley-space}). As with Boolean algebras, one may ask if Priestley duality can be lifted to the category $\PC$ of compact ordered spaces while retaining its algebraic nature. In~\cite{HNN2018}, the authors showed that $\PCop$ is equivalent to an $\aleph_1$-ary quasi-variety and partially described its algebraic theory. In the recent work~\cite{Abbadini2019}, the first-named author proved that $\PCop$ is in fact equivalent to an $\aleph_1$-ary variety, by providing an equational axiomatisation. The result is rather involved, and is based on an algebraic language whose finitary reduct extends the positive part of the language of MV-algebras~\cite{cdm2000}.

In this note we provide a new proof of the fact that $\PCop$ is equivalent to an $\aleph_1$-ary variety, which relies only on properties of Nachbin's compact ordered spaces. The structure of our proof is the following. A well-known result in category theory, recalled in Section~\ref{s:varieties-as-categories}, characterises those categories which are equivalent to some variety of possibly infinitary algebras. A key property, which distinguishes varieties among quasi-varieties, is the effectiveness of (internal) equivalence relations. In Section~\ref{s:compact-pospaces} we recall some basic facts about compact ordered spaces. Further, we state Theorem~\ref{th:effective}, asserting that equivalence relations in $\PCop$ are effective, and show that it implies that $\PCop$ is equivalent to an $\aleph_1$-ary variety. Sections~\ref{s:equivalence-co-relations}--\ref{s:proof-of-main-res} contain the proof of Theorem~\ref{th:effective}. First, we characterise equivalence relations on a compact ordered space $X$, in the category $\PCop$, as certain pre-orders on the order-topological coproduct $X+X$. Then, we rephrase effectiveness into an order-theoretic condition and show that it is satisfied by every pre-order arising from an equivalence relation. Finally, in Section~\ref{s:epilogue}, we show that the bound $\aleph_1$ is best possible: $\PC$ is not dually equivalent to any class of finitary algebras which is closed under taking subalgebras and Cartesian products.

\begin{notation}
Given morphisms $f_i\colon X\to Y_i$ for $i\in\{0,1\}$, the unique morphism induced by the universal property of the product is $\langle f_0, f_1\rangle\colon X\to Y_0\times Y_1$. Similarly, given morphisms $g_i\colon X_i\to Y$ with $i\in\{0,1\}$, the coproduct map is $\binom{g_0}{g_1}\colon X_0+X_1\to Y$. For infinite coproducts, we use the notation $\sum_{i\in I}{X_i}$. Epimorphisms are denoted by $\epi$, while monomorphisms (resp.\ regular monomorphisms) by $\mono$ (resp.\ $\rmono$). We use the symbol $\les$ for pre-orders, and $\leq$ for partial orders.
\end{notation}

\section{Varieties as categories}\label{s:varieties-as-categories}
In this section we provide the background needed to state a well-known characterisation of those categories which are equivalent to some (quasi-)variety of algebras. See Theorem~\ref{th:char-quasi-varieties} below. Throughout, all categories are assumed to be locally small and, unless otherwise stated, (quasi-)varieties admit possibly infinitary function symbols in their signatures.

Recall from~\cite{Borceux-vol2} or~\cite{BGvO} that a category $\C$ is \emph{regular} provided \emph{(i)} it has finite limits, \emph{(ii)} it admits coequalisers of kernel pairs, and \emph{(iii)} regular epimorphisms in $\C$ are stable under pullbacks. For instance, varieties and quasi-varieties of algebras (with homomorphisms) are regular categories. Those regular categories in which there is a good correspondence between regular epimorphisms and equivalence relations are called \emph{exact}. In order to give a precise definition, we recall the notion of equivalence relation in a category.

Let $\C$ be a category with finite limits and $A$ an object of $\C$. An \emph{(internal) equivalence relation} on $A$ is a subobject 
$\langle p_0,p_1\rangle\colon R\mono A\times A$ satisfying the following properties:
\begin{description}[leftmargin=*,labelindent=8pt]
\item[reflexivity] there exists a morphism $d\colon A\to R$ in $\C$ such that the following diagram commutes;
\[\begin{tikzcd}
A\arrow[swap,tail]{rd}{\langle 1_A,1_A\rangle}\arrow[dashed]{rr}{d} & & R\arrow[tail]{dl}{\langle p_0,p_1\rangle}\\
& A\times A &
\end{tikzcd}\]
\item[symmetry] there exists a morphism $s\colon R\to R$ in $\C$ such that the following diagram commutes;
\[\begin{tikzcd}
R\arrow[dashed]{rr}{s}\arrow[tail,swap]{dr}{\langle p_1,p_0\rangle} & & R\arrow[tail]{dl}{\langle p_0,p_1\rangle}\\
& A\times A &
\end{tikzcd}\]
\item[transitivity] if the left-hand diagram below is a pullback square in $\C$,
\[\begin{tikzcd}
P \arrow{r}{\pi_1} \arrow[swap]{d}{\pi_0}  \arrow[dr, phantom, "\ulcorner", very near start] & R \arrow{d}{p_0} \\
R \arrow[swap]{r}{p_1} & A
\end{tikzcd}
 \ \ \ \ \ \ 
\begin{tikzcd}
P\arrow[swap]{rd}{\langle p_0\circ \pi_0,p_1\circ\pi_1\rangle}\arrow[dashed]{rr}{t}&&R\arrow[tail]{dl}{\langle p_0,p_1\rangle}\\
& A\times A &
\end{tikzcd}
\]
then there is a morphism $t\colon P\to R$ such that the right-hand diagram commutes.
\end{description}
\begin{definition}\label{d:effective-exact}
An equivalence relation $\langle p_0,p_1\rangle\colon R\mono A\times A$ is \emph{effective} if it coincides with the kernel pair of the coequaliser of $p_0$ and $p_1$. A regular category $\C$ is \emph{exact} if every equivalence relation in $\C$ is effective.
\end{definition}
For categories of algebras, the definition of equivalence relation given above coincides with the usual notion of congruence. Varieties of algebras are therefore exact categories, while the effective equivalence relations in quasi-varieties are the so-called \emph{relative} congruences.

We need one last piece of terminology to state the desired characterisation of (quasi-)varieties of algebras. 
Recall that an object $G$ of a category $\C$ is a \emph{regular generator} if \emph{(i)} for every set $I$ the copower $\sum_I{G}$ exists in $\C$, and \emph{(ii)} for every object $A$ of $\C$, the canonical morphism
\[
\sum_{\hom_{\C}(G,A)}{G}\to A
\]
is a regular epimorphism. Further, $G$ is \emph{regular projective} if, for every morphism $f\colon G\to A$ and regular epimorphism $g\colon B\to A$, $f$ factors through $g$. We can now state the following well-known result.

\begin{theorem}\label{th:char-quasi-varieties}
For a category $\C$, consider the following conditions:
\begin{enumerate}
\item $\C$ is regular with coequalisers of equivalence relations;
\item $\C$ has a regular projective regular generator $G$;
\item every equivalence relation in $\C$ is effective.
\end{enumerate}
The category $\C$ is equivalent to a quasi-variety if and only if it satisfies $1$ and $2$, and it is equivalent to a variety if and only if it satisfies $1$, $2$ and $3$.\qed
\end{theorem}
The abstract characterisation of varieties and quasi-varieties has a long history in category theory, starting with the works of Lawvere, Isbell, Linton, Felscher and Duskin in the 1960s. We do not attempt here to provide an accurate historical account. For the case of quasi-varieties we refer the reader to \cite[Theorem~1.8]{Pedicchio1996}, and for varieties to \cite[Theorem~4.4.5]{Borceux-vol2} or~\cite{Vitale1994}. Further, we point out that the assumption that $\C$ be regular can be omitted provided $\C$ has all coequalisers, cf.\ \cite[Theorem~3.6]{Adamek2004}.

\section{Compact ordered spaces and their dual variety}\label{s:compact-pospaces}
We collect here some basic facts about compact ordered spaces, first introduced by Nachbin~\cite{Nachbin}. In particular, we describe their limits and colimits. This will come handy in the following sections.
\begin{definition}\label{def:compact-pospace}
A \emph{compact ordered space} (or \emph{compact pospace}, for short) is a pair $(X,\leq)$ where $X$ is a compact space and $\leq$ is a partial order on $X$ which is closed in the product topology of $X\times X$. We write $\PC$ for the category of compact pospaces and continuous monotone maps.  
\end{definition}
A basic example of compact pospace is the unit interval $\int$ equipped with the Euclidean topology and its usual total order.
Note that, for any compact pospace $(X,\leq)$, the opposite order ${\leq^\op}=\{(x,y)\mid y\leq x\}$ is also closed in the product topology of $X\times X$. The intersection ${\leq}\cap{\leq^\op}$ coincides with the diagonal $\Delta_X=\{(x,x)\mid x\in X\}$, which is thus closed in $X\times X$. That is, $X$ is a Hausdorff space. 

This shows that there is a forgetful functor $\PC\to\KH$, where $\KH$ denotes the category of compact Hausdorff spaces and continuous maps. On the other hand, there is also a functor $\Delta\colon \KH\to\PC$ sending a compact Hausdorff space $X$ to the compact pospace $(X,\Delta_X)$. It is readily seen that $\Delta$ is left adjoint to the forgetful functor $\PC\to\KH$. In symbols,
\begin{equation}\label{eq:adjoint-Lambda}
\begin{tikzcd}
\PC \arrow[yshift=5pt]{rr} & {\adj} & \KH \arrow[yshift=-5pt]{ll}{\Delta}.
\end{tikzcd}
\end{equation}

We will see in a moment that $\PC$ admits all limits and colimits. By the adjunction in~\eqref{eq:adjoint-Lambda}, limits in $\PC$ are computed in $\KH$, whence in the category of sets. However, this is not the case for colimits. To circumvent this issue, we embed $\PC$ in a larger category where colimits admit a simpler description.
\begin{definition}
A \emph{pre-ordered compact Hausdorff space} is a pair $(X,\les)$ where $X$ is a compact Hausdorff space and $\les$ is a pre-order on $X$ which is closed in the product topology of $X\times X$. We write $\PreC$ for the category of pre-ordered compact Hausdorff spaces and continuous monotone maps.  
\end{definition}

Clearly, $\PC$ is a full subcategory of $\PreC$ and the adjunction in~\eqref{eq:adjoint-Lambda} lifts to an adjunction between $\PreC$ and $\KH$. Further, the forgetful functor $\PreC\to\KH$ has, in addition to the left adjoint $\Delta$, also a right adjoint. Write $\nabla\colon \KH\to\PreC$ for the functor sending a compact Hausdorff space $X$ to the pre-ordered compact Hausdorff space $(X,\nabla_X)$, where $\nabla_X=X\times X$ is the improper relation on $X$. It is immediate that $\nabla$ is right adjoint to the forgetful functor $\PreC\to\KH$.

Given a pre-ordered compact Hausdorff space $(X,\les)$, we can consider the quotient of $X$ with respect to the symmetrization of $\les$, that is the equivalence relation ${\sim}={\les}\cap {\les^\op}$. The pre-order $\les$ descends to a partial order $\leq$ on the quotient space $X/{\sim}$, and the map
\[
\rho_X\colon (X,\les)\epi (X/{\sim},\leq)
\]
is continuous and monotone. The pair $(X/{\sim},\leq)$ is readily seen to be a compact pospace. This assignment extends to a functor $\refl\colon \PreC\to\PC$ which is left adjoint to the inclusion $\PC\to\PreC$. In other words, $\PC$ is a reflective subcategory of $\PreC$.
\[\begin{tikzcd}
\PC \arrow[yshift=5pt]{rr} & \adj & \PreC \arrow{rr} \arrow[yshift=-5pt]{ll}{\refl} & & \KH \arrow[yshift=7pt]{ll}[outer sep=-0.5pt]{\scriptscriptstyle\top}[swap]{\nabla} \arrow[yshift=-7pt]{ll}{\Delta}[swap, outer sep=-0.2pt]{\scriptscriptstyle\top}
\end{tikzcd}\] 
The category $\PreC$ is complete and cocomplete because the forgetful functor $\PreC\to\KH$ is topological~\cite[Example~2]{Tholen2009}, hence so is its reflective subcategory $\PC$. Since the forgetful functor $\PreC\to\KH$ has a right adjoint, colimits in $\PreC$ are computed in $\KH$. In turn, the colimit of a diagram in $\PC$ can be obtained by first computing the colimit in $\PreC$, and then applying the reflector $\refl$.
For more details, cf.\ Remark~\ref{rm:pushouts}.
\begin{proposition}\label{p:properties-of-PC}
The following statements hold:
\begin{enumerate}
\item the regular monomorphisms in $\PC$ are the continuous order-embeddings;
\item the epimorphisms in $\PC$ are the continuous monotone surjections;
\item the unit interval $\int$ is a regular injective regular cogenerator in $\PC$.
\end{enumerate}
\end{proposition}
\begin{proof}
See, e.g., \cite[Theorem~2.6]{HNN2018}. In particular, the unit interval is a regular cogenerator in $\PC$ by \cite[Chapter~I, Theorems~1 and~4]{Nachbin}, and it is regular injective by \cite[Chapter~I, Theorem~6]{Nachbin}.
\end{proof}
\begin{remark}\label{rm:pushouts}
Let $X, Y$ be compact pospaces. Their coproduct in $\PreC$ is the disjoint union of $X$ and $Y$, with the coproduct topology and the coproduct pre-order. The latter is a compact pospace, whence it coincides with the coproduct of $X$ and $Y$ in $\PC$.
Next, we describe certain pushouts in $\PC$. 

Consider regular monomorphisms $f_0\colon X\rmono Y_0$, $f_1\colon X\rmono Y_1$ in $\PC$ and their pushout in the category $\PreC$, as displayed in the following diagram.
\[\begin{tikzcd}
X \arrow[hookrightarrow]{r}{f_1} & Y_1 \arrow{d}{\lambda_0} \\
Y_0 \arrow[hookleftarrow]{u}{f_0} \arrow[swap]{r}{\lambda_1} & P \arrow[ul, phantom, "\lrcorner", very near start]
\end{tikzcd}\]
As a space, $P$ is homeomorphic to the quotient of the coproduct space $Y_0+Y_1$ with respect to the equivalence relation
\[
\{(f_i(x),f_j(x))\in (Y_0+Y_1)\times (Y_0+Y_1) \mid x\in X, \ i,j \in \{0,1\}\}
\]
(transitivity follows because $f_0$ and $f_1$ are order-embeddings by item~$1$ in Proposition~\ref{p:properties-of-PC}), which is easily seen to be closed in the product topology.
Let $i\in\{0,1\}$, and write $i^*=1-i$. With this notation, the pre-order on $P$ is given by $\Theta=\Theta'\cup\Theta''$, where
\begin{gather*}
\Theta'=\{(p,q)\in P\times P\mid \exists i\in\{0,1\}, \ \exists w\in \lambda_{i}^{-1}(p), \ \exists w'\in\lambda_{i}^{-1}(q), \ w\leq_{Y_i} w'\}, \\
\Theta''=\{(p,q)\in P\times P\mid \exists i\in\{0,1\}, \ \exists w\in \lambda_{i^*}^{-1}(p), \ \exists w'\in\lambda_{i}^{-1}(q), \ \exists x\in X, \ w\leq_{Y_i} f_i(x) \text{ and } f_{i^*}(x)\leq_{Y_{i^*}} w'\}.
\end{gather*}
The relation $\Theta$ is clearly reflexive, and is seen to be transitive again by item~$1$ in Proposition~\ref{p:properties-of-PC}. Note that $\Theta'=\bigcup_{i\in\{0,1\}}{(\lambda_i\times \lambda_i)(\leq_{Y_i})}$ is closed in $P\times P$. On the other hand,
\begin{align*}
\Theta''&=\bigcup_{i\in\{0,1\}}{(\lambda_{i^*}\times \lambda_i)(\{(w,w')\in Y_i\times Y_{i^*}\mid \exists x\in X, \ w\leq_{Y_i} f_i(x), \ f_{i^*}(x)\leq_{Y_{i^*}}w'\})} \\
&=\bigcup_{i\in\{0,1\}}{(\lambda_{i^*}\times \lambda_i)(\{(w,w')\in Y_i\times Y_{i^*}\mid \exists x\in X, (w,w')\leq_{Y_i\times Y_{i^*}^\partial} (f_i(x),f_{i^*}(x))\})} \\
&=\bigcup_{i\in\{0,1\}}{(\lambda_{i^*}\times \lambda_i)(\downset \image(X\xrightarrow{\langle f_i,f_{i^*}\rangle} Y_i\times Y_{i^*}^\partial))},
\end{align*}
where $Y_{i^*}^\partial=(Y_{i^*},\leq_{Y_{i^*}}^\op)$.
Since the downward closure $\downset D$ of any closed subset $D$ of a compact pospace is again closed~\cite[Proposition~4]{Nachbin}, we conclude that $\Theta''$ is also closed. Whence, $\Theta$ is a closed pre-order. It is not difficult to see that it is the smallest pre-order on $P$ making $\lambda_0$ and $\lambda_1$ monotone. Finally, the pushout of $f_0$ along $f_1$ in $\PC$ is obtained by applying the reflector $\refl\colon \PreC\to\PC$ to $P$.
\end{remark}
\begin{corollary}\label{cor:PCop-quasi-variety}
The category $\PC$ is dually equivalent to a quasi-variety of algebras.
\end{corollary}
\begin{proof}
By Theorem~\ref{th:char-quasi-varieties}, it is enough to show that \emph{(i)} $\PCop$ is regular with coequalisers of equivalence relations, and \emph{(ii)} it admits a regular projective regular generator $G$. 

We already observed that $\PC$ is complete and cocomplete. Whence, so is $\PCop$. To show that $\PCop$ is regular, it suffices to prove that regular monos, i.e.\ continuous order-embeddings, are stable under pushouts in $\PC$. Pushouts in $\PC$ can be computed by first taking the pushout in $\PreC$, and then composing with the reflection map. Reasoning as in Remark~\ref{rm:pushouts}, it is not difficult to see that the pushout of a continuous order-embedding in $\PreC$ is again a continuous order-embedding. Further, composing with the reflection yields again a continuous order-embedding, i.e.\ a regular mono in $\PC$. This proves \emph{(i)}. In turn, \emph{(ii)} follows at once from item $3$ in Proposition~\ref{p:properties-of-PC}, by setting $G=\int$. We mention that one may also deduce \emph{(i)} from \emph{(ii)} and the fact that $\PCop$ is cocomplete, cf.~\cite[Theorem~3.6]{Adamek2004}.
\end{proof}
The latter fact was already observed in~\cite{HNN2018} where, in addition, the authors provide a description of an $\aleph_1$-ary quasi-variety dually equivalent to $\PC$ (see Theorem~3.15 in \emph{op.\ cit.}). Our main contribution consists in a direct proof of the following result:
\begin{theorem}\label{th:effective}
Every equivalence relation in $\PCop$ is effective.
\end{theorem}
A proof of the previous theorem is provided in Sections~\ref{s:equivalence-co-relations}--\ref{s:proof-of-main-res}. We conclude this section by observing that Theorem~\ref{th:effective} implies that $\PCop$ is equivalent to an $\aleph_1$-ary variety of algebras, that is a variety of algebras in a language consisting of function symbols of at most countably infinite arity.
\begin{corollary}
The category $\PC$ is dually equivalent to an $\aleph_1$-ary variety of algebras.
\end{corollary}
\begin{proof}
By Corollary~\ref{cor:PCop-quasi-variety}, we know that $\PC$ is dually equivalent to a quasi-variety of algebras.  Theorems~\ref{th:char-quasi-varieties} and~\ref{th:effective} entail that $\PC$ is in fact dually equivalent to a variety of algebras. Indeed, $\PCop$ is equivalent to the category of Eilenberg-Moore algebras for the monad induced by the adjunction
\[
\sum_{-}{\int}\dashv \hom_{\PCop}(\int,-)\colon \PCop\to \Set.
\]
The latter is equivalent to an $\aleph_1$-ary variety of algebras if, and only if, the monad preserves $\aleph_1$-directed colimits. It suffices to show that, for every set $I$ and continuous monotone function $f\colon \int^I\to \int$, there is a countable subset $J\subseteq I$ such that $f$ factors through the projection $\int^I\epi\int^J$. In turn, this is a consequence of the Stone-Weierstrass Theorem for compact Hausdorff spaces~\cite{Stone1937}.
\end{proof}

\section{Equivalence co-relations on compact ordered spaces}\label{s:equivalence-co-relations}
In this section we provide a description of equivalence relations in the category $\PCop$, which will then be exploited in the next section to prove that equivalence relations in $\PCop$ are effective.

To start with, we dualise the notion of subobject. Given a compact pospace $X$, a \emph{quotient object} of $X$ is a subobject of $X$ in the category $\PCop$. The poset of quotient objects of $X$ is denoted by $\Quot(X)$.
Explicitly, $\Quot(X)$ is the poset of (equivalence classes of) epimorphisms with domain $X$, where $f_1\colon X\epi Y_1$ is below $f_2\colon X\epi Y_2$ whenever there exists $g\colon Y_2\to Y_1$ such that $g\circ f_2=f_1$.
\[
\begin{tikzcd}
X \arrow[swap,two heads]{rd}{f_2} \arrow[two heads]{r}{f_1} & Y_1 \\
& Y_2 \arrow[dashed]{u}[swap]{g}
\end{tikzcd}
\]
\begin{remark}
We warn the reader that our terminology is non-standard. By a quotient object we do not mean a regular epimorphism, but what may be called a \emph{co-subobject} (not every epimorphism in $\PC$ is regular).
\end{remark}

By definition, an equivalence relation on $X$ in the opposite category $\PCop$ is a subobject of $X\times X$ (where the product is computed in $\PCop$) which is reflexive, symmetric and transitive. This corresponds to a quotient object $\binom{q_0}{q_1}\colon X+ X\epi S$ of the compact pospace $X+X$ satisfying the dual properties:

\begin{minipage}[t]{0.5\textwidth}
\begin{figure}[H]
\centering
\begin{tikzcd}
{} & X+X \arrow[two heads,swap]{dl}{\binom{q_0}{q_1}} \arrow[two heads]{dr}{\binom{1_X}{1_X}} & \\
S\arrow[dashed,swap]{rr}{d} & & X
\end{tikzcd}
{\vspace{-5pt}\caption*{co-reflexivity}}
\end{figure}
\end{minipage}
\begin{minipage}[t]{0.5\textwidth}
\begin{figure}[H]
\centering
\begin{tikzcd}
& X+X \arrow[swap, two heads]{dl}{\binom{q_0}{q_1}} \arrow[two heads]{dr}{\binom{q_1}{q_0}} & \\
S \arrow[dashed,swap]{rr}{s} & & S
\end{tikzcd}
{\vspace{-5pt}\caption*{co-symmetry}}
\end{figure}
\end{minipage}
\begin{figure}[H]
\centering
\begin{tikzcd}
X \arrow{r}{q_0} \arrow[swap]{d}{q_1} & S \arrow{d}{\lambda_1} \\
S \arrow[swap]{r}{\lambda_0} & P \arrow[ul, phantom, "\lrcorner", very near start]
\end{tikzcd}
\ \ \ \ $\Longrightarrow$ \ \ \ \
\begin{tikzcd}
& X+X \arrow[swap, two heads]{dl}{\binom{q_0}{q_1}} \arrow{dr}{\binom{\lambda_0\circ q_0}{\lambda_1\circ q_1}} & \\
S\arrow[dashed,swap]{rr}{t} & & P
\end{tikzcd}
{\vspace{-5pt}\caption*{co-transitivity}}
\end{figure}
\noindent A quotient object of $X+X$ which satisfies the three properties above will be called an \emph{equivalence co-relation} on $X$.
The key observation is that equivalence co-relations are more manageable than their duals, because quotient objects of $X$ are in bijection with certain pre-orders on $X$. 

Indeed, if $f\colon (X,\leq_X)\epi (Y,\leq_Y)$ is an epimorphism in $\PC$, then
\begin{equation*}
{\les_f}=\{(x_1,x_2)\in X\times X\mid f(x_1)\leq_Y f(x_2)\}
\end{equation*}
is a pre-order on $X$. The monotonicity of $f$ entails ${\leq_X}\subseteq{\les_f}$. 
Further, recalling that epimorphisms in $\PC$ are precisely the continuous monotone surjections (see item~$2$ in Proposition~\ref{p:properties-of-PC}), we see that $\les_f$ is closed in $X\times X$ because it coincides with the preimage of $\leq_Y$ under the continuous map $f\times f\colon X\times X\to Y\times Y$. 
Let us denote by $\Preo(X)$ the poset of all closed pre-orders on $X$ which extend $\leq_X$, ordered by reverse inclusion. By the previous discussion, there is a map $\Quot(X)\to \Preo(X)$ sending $f$ to $\les_f$. This function is well-defined, as $\les_f$ does not depend on the choice of a representative in the equivalence class of $f$. 
Conversely, given a pre-order $\les$ in $\Preo(X)$, consider its symmetrization ${\sim}={\les}\cap{\les^\op}$. 
The space $X/{\sim}$, equipped with the quotient topology, is compact. The direct image of $\les$ under the quotient map is a partial order on $X/{\sim}$, and it is closed because so is $\les$. Moreover, since ${\leq_X}\subseteq{\les}$, we get an epimorphism \[X\epi X/{\sim}\] in $\PC$. Taking its equivalence class, we obtain an element of $\Quot(X)$. The following fact follows easily.
\begin{lemma}\label{l:bijection Q E}
For every compact pospace $X$, the assignments
\[ (f\colon X\epi Y) \ \mapsto \ {\les_f} \ \ \text{ and } \ \ {\les} \ \mapsto \ (X\epi X/{\sim})\]
induce an isomorphism between the posets $\Preo(X)$ and $\Quot(X)$.\qed
\end{lemma}
\begin{remark}\label{r:bijection not needed surj}
Assume $f_1\colon X\to Y_1$ and $f_2\colon X\to Y_2$ are surjective morphisms in $\PC$. By Lemma~\ref{l:bijection Q E}, there exists $g\colon Y_1\to Y_2$ such that $g\circ f_1=f_2$ if, and only if, $\forall x,y\in X$, $f_1(x)\leq f_1(y)$ implies $f_2(x)\leq f_2(y)$. In fact, it is not difficult to see that this is true even if $f_2$ is not surjective, as we can factor it as a surjective map followed by an injective one.
\end{remark}
Recall from Remark~\ref{rm:pushouts} that the compact pospace $X+X$ is isomorphic to the disjoint union
\[
\{(x,0)\mid x\in X\}\cup \{(x,1)\mid x\in X\},
\]
equipped with the coproduct topology and the coproduct order.
\begin{notation}
We denote the elements of $X+X$ by $(x,i),(y,j),\ldots$ where $i,j$ vary in $\{0,1\}$. Further, $i^*$ stands for $1-i$. For example, $(x,1^*)=(x,0)$.
\end{notation}
For the rest of this section, we fix a quotient object $\binom{q_0}{q_1}\colon X+X\epi S$ of a compact pospace $X$. We write $\les_{\binom{q_0}{q_1}}$, or simply $\les_S$, for the associated pre-order on $X+X$. 
We say that $\les_S$ is \emph{co-reflexive} (\emph{co-symmetric}, \emph{co-transitive}) if so is $\binom{q_0}{q_1}$. 
To improve readability, we write $[(x,i)]$ instead of $\binom{q_0}{q_1}(x,i)$.

\begin{lemma}\label{l:co-refl-symm}
The following statements hold:
\begin{enumerate}
\item the pre-order $\les_S$ is co-reflexive if, and only if, $(x,i)\les_S(y,j)$ entails $x\leq y$;
\item the pre-order $\les_S$ is co-symmetric if, and only if, $(x,i)\les_S(y,j)$ entails $(x,i^*)\les_S (y,j^*)$.
\end{enumerate}
\end{lemma}
\begin{proof}
\begin{enumerate}[wide, labelwidth=!, labelindent=5pt]
\item By definition, $\les_S$ is co-reflexive if, and only if, $\binom{q_0}{q_1}\colon X+X\epi S$ is above $\binom{1_X}{1_X}\colon X+X\epi X$ in the poset $\Quot(X+X)$. By Lemma~\ref{l:bijection Q E}, this is equivalent to ${\les_S} \seq {\les_{\binom{1_X}{1_X}}}$. Given $(x,i),(y,j)\in X+X$, we have
\[
(x,i)\les_{\binom{1_X}{1_X}}(y,j) \ \Longleftrightarrow \ x\leq y.
\]
It follows that the pre-order $\les_S$ is co-reflexive if, and only if, $(x,i)\les_S(y,j)$ entails $x\leq y$.
\item Again, by definition, $\les_S$ is co-symmetric if and only if  $\binom{q_0}{q_1}\colon X+X\epi S$ is above $\binom{q_1}{q_0}\colon X+X\epi S$ in $\Quot(X+X)$. By Lemma~\ref{l:bijection Q E}, this happens exactly when ${\les_S} \seq {\les_{\binom{q_1}{q_0}}}$. Given $(x,i),(y,j)\in X+X$,
\[
(x,i)\les_{\binom{q_1}{q_0}}(y,j) \ \Longleftrightarrow \ (x,i^*)\les_S (y,j^*).
\]
Therefore, the pre-order $\les_S$ is co-symmetric if, and only if, $(x,i)\les_S(y,j)$ entails $(x,i^*)\les_S (y,j^*)$.\qedhere
\end{enumerate}
\end{proof}
\begin{lemma}\label{l:co-transitive}
Assume the pre-order $\les_S$ is co-reflexive. Then it is co-transitive if, and only if, 
\[
(x,i)\les_S(y,i^*) \ \Longrightarrow \ \exists z\in X \ [(x,i)\les_S (z,i^*) \text{ and } (z,i)\les_S (y,i^*)].
\]
\end{lemma}
\begin{proof}
Recall that $\les_S$ is co-transitive if, and only if, given a pushout square in $\PC$ as in the left-hand diagram below,
\[\begin{tikzcd}
X \arrow{r}{q_0} \arrow[swap]{d}{q_1} & S \arrow{d}{\lambda_1} \\
S \arrow[swap]{r}{\lambda_0} & P \arrow[ul, phantom, "\lrcorner", very near start]
\end{tikzcd}
 \ \ \ \ \ \ \ \ 
\begin{tikzcd}
{} & X+X \arrow[swap, two heads]{dl}{\binom{q_0}{q_1}} \arrow{dr}{\binom{\lambda_0\circ q_0}{\lambda_1\circ q_1}} & \\
S \arrow[dashed,swap]{rr}{t} & & P
\end{tikzcd}
\]
there is $t\colon S\to P$ making the right-hand diagram commute.
By Remark~\ref{r:bijection not needed surj}, such a $t$ exists precisely when, for every $(x,i),(y,j)\in X+X$, $(x,i)\les_S(y,j)$ implies $\binom{\lambda_0\circ q_0}{\lambda_1\circ q_1}(x,i)\leq \binom{\lambda_0\circ q_0}{\lambda_1\circ q_1}(y,j)$, i.e.\ $\lambda_i([(x,i)])\leq \lambda_j([(y,j)])$. Recall that $\les_S$ is co-reflexive provided $q_0$ and $q_1$ are both sections of a morphism $d\colon S\to X$. In particular, $q_0$ and $q_1$ are regular monomorphisms in $\PC$. Thus, by Remark~\ref{rm:pushouts}, $\lambda_i([(x,i)])\leq \lambda_j([(y,j)])$ if, and only if,
\begin{equation}\label{eq:preord-pushout}
[i=j \text{ and } (x,i)\les_S (y,j) ] \text{ or } [i\neq j \text{ and } \exists z\in X \text{ s.t.\ } (x,i)\les_S (z,j) \text{ and } (z,i)\les_S (y, j)].
\end{equation}
	
We conclude that $\les_S$ is co-reflexive if, and only if, equation~\eqref{eq:preord-pushout} holds whenever $(x,i)\les_S(y,j)$. In turn, this is equivalent to the condition in the statement of the lemma.
\end{proof}

From Lemmas~\ref{l:co-refl-symm} and~\ref{l:co-transitive}, we obtain the following characterisation of equivalence co-relations in $\PC$.
\begin{proposition}\label{p:co-equivalence}
The pre-order $\les_S$ is an equivalence co-relation on $X$ if, and only if,
\[
(x,i)\les_S (y,j) \ \Longrightarrow \ [x\leq y \text{ and } (x,i^*)\les_S (y,j^*)]
\]
and
\[\pushQED{\qed} 
(x,i)\les_S(y,i^*) \ \Longrightarrow \ \exists z\in X \ [(x,i)\les_S (z,i^*) \text{ and } (z,i)\les_S (y,i^*)].\qedhere
\popQED\]
\end{proposition}

\section{Proof of Theorem~\ref{th:effective}}\label{s:proof-of-main-res}
Assume $\binom{q_0}{q_1}\colon X+X\epi S$ is an equivalence co-relation on $X$. Dualising Definition~\ref{d:effective-exact}, we say that $\binom{q_0}{q_1}$ is \emph{effective} provided it coincides with the co-kernel pair of its equaliser. That is, provided the following is a pushout square in $\PC$,
\begin{equation}\label{eq:pushout-effective}
\begin{tikzcd}
Y \arrow[hookrightarrow]{r}{k} & X \arrow{d}{q_1} \\
X \arrow[hookleftarrow]{u}{k} \arrow[swap]{r}{q_0} & S
\end{tikzcd}\end{equation}
where $k$ is the equaliser of $q_0,q_1\colon X\rightrightarrows S$ in $\PC$. Also, we say that the pre-order $\les_S$ is effective if so is the corresponding quotient object.
By item~$1$ in Proposition~\ref{p:properties-of-PC}, the space $Y$ can be identified with a closed subset of $X$, equipped with the induced order and topology. Define the relation $\les^{Y}$ on $X+X$ as follows:
\begin{equation}\label{eq:preord-closed-sub}
(x,i)\les^{Y}(y,j) \ \Longleftrightarrow \ (i=j \text{ and } x\leq y) \text{ or } (i^*=j \text{ and } \exists z\in Y \text{ s.t.\ } x\leq z\leq y).
\end{equation}
\begin{lemma}\label{l:co-kernel-inclusion}
$\les^Y$ is the pre-order on $X+X$ associated with the pushout of the inclusion $Y\rmono X$ along itself.
\end{lemma}
\begin{proof}
This is an immediate consequence of Remark~\ref{rm:pushouts}.
\end{proof}
For the next proposition, recall that ${\sim_S}={\les_S}\cap{\les_S^\op}$ is the symmetrization of the pre-order $\les_S$.
\begin{proposition}\label{p:effective-char}
The equivalence co-relation $\les_S$ is effective if, and only if, 
\[(x,i)\les_S (y,i^*) \ \Longrightarrow \  \exists z\in X \ [x\leq z\leq y \text{ and } (z,i)\sim_S (z,i^*)].\]
\end{proposition}
\begin{proof}
Recall that the equivalence co-relation $\les_S$ is effective if, and only if, the diagram in~\eqref{eq:pushout-effective} is a pushout in $\PC$. In turn, by Lemma~\ref{l:co-kernel-inclusion}, this is equivalent to saying that ${\les_S}={\les^Y}$.
Since $Y=\{x\in X\mid (x,i)\sim_S (x,i^*)\}$,
\begin{equation*}
(x,i)\les^{Y}(y,j) \ \Longleftrightarrow \ \exists z\in X \ [x\leq z\leq y \text{ and }  (z,i)\sim_S (z,i^*)].
\end{equation*}
Therefore, to settle the statement, it suffices to show that the inclusion ${\les^Y}\subseteq{\les_S}$ is always satisfied.

Note that any equivalence co-relation $\les$ on $X$ satisfies $(x,i)\les (y,i)$ if, and only if, $x\leq y$. The left-to-right implication follows from item~$1$ in Lemma~\ref{l:co-refl-symm}, while the right-to-left implication holds because $\les$ extends the coproduct order of $X+X$. Whence, $(x,i)\les^Y (y,i)$ if, and only if, $(x,i)\les_S (y,i)$. Suppose now $(x,i)\les^Y (y,i^*)$, and let $z\in Y$ satisfy $x\leq z\leq y$. We have
\[
(x,i)\les_S (z,i)\sim_S (z,i^*)\les_S (y,i^*),
\]
where the two inequalities hold because $\les_S$ extends the partial order of $X+X$. Therefore, ${\les^Y}\subseteq{\les_S}$.
\end{proof}

We can finally prove Theorem~\ref{th:effective}, stating that every equivalence relation in $\PCop$ is effective.
 \begin{proof}[Proof of Theorem~\ref{th:effective}]
 Let $(X,\leq)$ be a compact pospace and $\les$ an equivalence co-relation on $X$. In view of Proposition~\ref{p:effective-char} it is enough to show that, whenever $(x,i)\les (y,i^*)$, there is $z\in X$ such that 
 \[
 x\leq z\leq y \text{ and } (z,i)\sim (z,i^*).
 \]
Fix arbitrary $x,y\in X$ and $i\in \{0,1\}$ satisfying $(x,i)\les (y,i^*)$, and set
\[
\Omega=\{u\in X\mid (x,i)\les (u,i^*) \text{ and } (u,i)\les (y,i^*) \}.
\]
 The idea is to apply Zorn's Lemma to show that $\Omega$ has a maximal element $z$ satisfying the desired properties. 
 
 First, note that $\Omega$ is non-empty because $\les$ is co-transitive, cf.\ Lemma~\ref{l:co-transitive}. We claim that every non-empty chain $C\subseteq \Omega$ admits an upper bound in $\Omega$. 
 Every directed set in a compact pospace has a supremum, which coincides with the topological limit of the set regarded as a net \cite[Proposition~VI.1.3]{ContLatt}. Thus, $C$ has a supremum $s$ in $X$, which belongs to the topological closure $\overline{C}$ of $C$. 
 \begin{claim*}
 $\Omega$ is a closed subset of $X$.
 \end{claim*}
 \begin{proof}
 The set $\Omega$ can be written as the intersection of the sets
 \[
 \Omega_1= \{u\in X\mid (x,i)\les (u,i^*) \} \text{ and } \Omega_2=\{u\in X\mid (u,i)\les (y,i^*) \}.
 \]
 Hence, it is enough to show that $\Omega_1$ and $\Omega_2$ are closed in $X$. We show that $\Omega_1$ is closed. The proof for $\Omega_2$ is the same, mutatis mutandis.
The set $\Omega_1$ is the preimage, under the coproduct injection $\iota_{i^*}\colon X\to X+X$, of 
 \[
 \upset(x,i)=\{(w,j)\in X+X\mid (x,i)\les (w,j) \}.
 \]
Since $\les$ is a closed pre-order on $X+X$, the set $\upset(x,i)$ is closed in $X+X$ \cite[Proposition~1]{Nachbin}. Therefore, its preimage $\Omega_1$ is closed in $X$.
 \end{proof}
 The previous claim entails that $s\in\overline{C}\subseteq\Omega$, i.e.\ $C$ has a supremum in $\Omega$. Hence, every non-empty chain in $\Omega$ admits an upper bound. By Zorn's Lemma, $\Omega$ has a maximal element $z$. By co-reflexivity of $\les$ (see item~$1$ in Lemma~\ref{l:co-refl-symm}), $(x,i)\les (z,i^*)$ and $(z,i)\les (y,i^*)$ imply $x\leq z\leq y$. 
 It remains to show that $(z,i)\sim (z,i^*)$.
 
 Since $(z,i)\les (y,i^*)$, by co-transitivity of $\les$ (cf.\ Lemma~\ref{l:co-transitive}), there is $u\in X$ such that $(z,i)\les (u,i^*)$ and $(u,i)\les (y,i^*)$. Also, $(x,i)\les(z,i)$ because $\les$ extends the partial order of $X$. Thus $(x,i)\les (z,i)\les (u,i^*)$, which implies $u\in \Omega$. By co-reflexivity, $(z,i)\les (u,i^*)$ entails $z\leq u$. Since $z$ is maximal, it must be $z=u$. Therefore, $(z,i)\les (z,i^*)$. By co-symmetry (see item~$2$ in Lemma~\ref{l:co-refl-symm}), we conclude that $(z,i)\sim (z,i^*)$.
 \end{proof}
 
 We saw that, for every compact pospace $X$ and closed subset $Y\subseteq X$, there is a pre-order $\les^Y$ on $X+X$ given as in~\eqref{eq:preord-closed-sub}. In fact, by Lemma~\ref{l:co-kernel-inclusion}, $\les^Y$ is the equivalence co-relation on $X$ associated with the pushout of the inclusion $Y\rmono X$ along itself.
 Conversely, every equivalence co-relation $\les$ on $X$ yields a closed subset of $X$, namely
 \[
 \fix(\les)=\{x\in X\mid (x,i)\sim (x,i^*)\}.
 \]
 \begin{corollary}\label{c:bijection congruences}
 For every compact pospace $X$, the assignments
 \[ {\les} \ \mapsto \ {\fix(\les)} \ \ \text{ and } \ \ {(Y\rmono X)} \ \mapsto \ {\les^Y}\]
yield an isomorphism between the poset of equivalence co-relations on $X$ and the poset of closed subsets of $X$.
 \end{corollary}
\begin{proof}
 The two maps are clearly monotone. For any closed subset $Y\seq X$, we have $\fix(\les^Y)=Y$ because
 \[
 (x,i)\sim^Y(x,i^*) \ \Longleftrightarrow \ \exists z\in Y \text{ s.t.\ } x\leq z\leq x \ \Longleftrightarrow \ x\in Y.
 \]
Moreover, it follows at once from Theorem~\ref{th:effective} and Proposition~\ref{p:effective-char} that, for any equivalence co-relation $\les$ on $X$, ${\les}\subseteq{\les^{\fix(\les)}}$. For the converse inclusion, see the proof of Proposition~\ref{p:effective-char}.
\end{proof}

 \section{Epilogue: negative axiomatisability results}\label{s:epilogue}
In the previous sections, we have given a direct proof of the fact that the category $\PC$ of compact ordered spaces is dually equivalent to an $\aleph_1$-ary variety of algebras. One may wonder whether it is necessary to resort to infinitary operations. In this section we show that $\PCop$ is not equivalent to any $\SP$-class of finitary algebras (i.e., one closed under subalgebras and Cartesian products), let alone a finitary (quasi-)variety.

Recall that an object $A$ of a category $\C$ is (\emph{Gabriel-Ulmer}) \emph{finitely presentable} if the covariant hom-functor $\hom_{\C}(A,-)\colon\C\to\Set$ preserves directed colimits. See \cite[Definition~6.1]{GabUlm} or \cite[Definition~1.1]{AdaRos}. Further, $\C$  is \emph{finitely accessible} provided it has directed colimits and there exists a set $S$ of its objects such that \emph{(i)} each object of $S$ is finitely presentable, and \emph{(ii)} each object of $\C$ is a directed colimit of objects in $S$. See \cite[Definition~2.1]{AdaRos}.  
For example, finitary varieties and finitary quasi-varieties (with homomorphisms) are finitely accessible categories, cf.\ \cite[Corollary~3.7 and Theorem~3.24]{AdaRos}.
Recall the following definition:
\begin{definition}\label{def:Priestley-space}
A \emph{Priestley space} is a compact pospace $(X,\leq)$ which is \emph{totally order-disconnected}, i.e.\ for all $x,y$ with $x\not\leq y$ there is a clopen $C\subseteq X$ which is an up-set for $\leq$ and satisfies $x\in C$ but $y\notin C$. 
\end{definition}
\begin{lemma}\label{l:pro-completion-posfin}
A compact pospace is a Priestley space if, and only if, it is the codirected limit in $\PC$ of finite posets equipped with the discrete topologies.
\end{lemma}
\begin{proof}
This result is folklore. For a proof see, e.g., \cite[Corollary~VI.3.3]{Johnstone1986}.
\end{proof}
Denote by $\Pr$ the full subcategory of $\PC$ defined by all Priestley spaces.
 By a result of Priestley~\cite{Priestley1970}, $\Pr^\op$ is equivalent to the category of bounded distributive lattices with homomorphisms. In particular, it is finitely accessible. The following result is an adaptation of \cite[Proposition~1.2]{MarraReggio2017} to the ordered case.
 \begin{theorem}\label{t:finitely accessible}
Let $\F$ be a full subcategory of $\PC$ extending $\Pr$. If $\F^\op$ is a finitely accessible category --- in particular, if $\F^\op$ is a finitary variety or a finitary quasi-variety --- then $\F= \Pr$.
 \end{theorem}
 \begin{proof}
 It suffices to show that every object in $\F$ is a Priestley space.
 We claim that every finitely copresentable object in $\F$ (i.e.\ one which is finitely presentable when regarded as an object of $\F^\op$) is finite. 

Let $(X,\leq)$ be an arbitrary finitely copresentable object in $\F$. Consider an epimorphism $\gamma\colon Y\epi X$ in $\PC$ with $Y$ a Priestley space. (E.g., let $Y=\beta|X|$ be the \v{C}ech-Stone compactification of the underlying set of $X$ equipped with the discrete topology, and $\gamma\colon (\beta|X|,=)\to (X,\leq)$ the unique continuous extension of the identity function $|X|\to |X|$). 
By Lemma~\ref{l:pro-completion-posfin}, $Y$ is the codirected limit in $\PC$ of finite posets $\{Y_i\}_{i\in I}$ with the discrete topologies. Denote by $\alpha_i\colon Y \to Y_i$ the $i$-th limit arrow. Since $Y$ lies in $\F$, and the full embedding $\F \to \PC$ reflects limits, $Y$ is in fact the codirected limit of $\{Y_i\}_{i\in I}$ in $\F$.
 \[
 \begin{tikzcd}
 Y \arrow[twoheadrightarrow]{r}{\gamma} \arrow{d}[swap]{\alpha_j} & X \\
 Y_j \arrow{ur}[swap]{\phi}
 \end{tikzcd}
 \]
The object $X$ being finitely copresentable in $\F$, there are $j\in I$ and a morphism $\phi\colon Y_j \to X$ such that $\gamma=\phi\circ \alpha_j$. The map $\gamma$ is surjective, hence so is $\phi$. This shows that $X$ is finite, and thus the claim is settled.
 
Since $\F^\op$ is finitely accessible, every object of $\F$ is the codirected limit of finitely copresentable objects. Using again the fact that the full embedding $\F \to \PC$ reflects limits, we deduce from Lemma~\ref{l:pro-completion-posfin} that every object of $\F$ is a Priestley space, as was to be shown.
Finally, we have already observed that finitary varieties and finitary quasi-varieties are finitely accessible categories.
 \end{proof}

 \begin{corollary}
 $\PCop$ is not equivalent to any $\SP$-class of finitary algebras.
 \end{corollary}
 \begin{proof}
 By Theorem~\ref{th:effective}, every equivalence relation in $\PCop$ is effective. In turn, Banaschewski observed in~\cite{Ban-var} that every $\SP$-class of finitary algebras in which every equivalence relation is effective is a variety of algebras. The statement then follows from Theorem~\ref{t:finitely accessible}.
 \end{proof}
 
 \begin{remark}
In a recent work, Lieberman, Rosick{\' y} and Vasey \cite{LRVpreprint} proved that the opposite of the category $\KH$ of compact Hausdorff spaces is not equivalent to any elementary class of structures, with morphisms all the homomorphisms. In fact, they show that there exists no faithful functor $\KH^\op\to\Set$ which preserves directed colimits. Since directed colimits in elementary classes are concrete~\cite{Richter1971}, the preceding statement follows. This implies that $\PCop$ is not equivalent to any elementary class of structures. Indeed, note that the embedding $\Delta\colon \KH^\op\to\PCop$ (cf.\ equation~\eqref{eq:adjoint-Lambda}) preserves directed colimits. Hence, if there were a faithful functor $F\colon \PCop\to \Set$ preserving directed colimits, the composition $F\circ \Delta\colon \KH^\op\to \Set$ would also be a faithful functor preserving directed colimits, contradicting the aforementioned result. This shows that $\PCop$ cannot be equivalent to an elementary class of structures with morphisms all the homomorphisms.
 \end{remark}

\noindent\textbf{Acknowledgements.} 
The first-named author expresses his gratitude to J.\ Somaglia for some helpful discussions on convergence of nets.
The second-named author is grateful to the Institute of Computer Science of the Czech Academy of Sciences and the Mathematical Institute of the University of Bern, where this work was carried out, for providing a highly stimulating environment. Finally, we are most grateful to the anonymous referee for their comments, which helped us to improve the presentation of our results.

\bibliographystyle{plain}

\end{document}